\documentclass[11pt]{article}
\usepackage{amsmath}
\usepackage{amsthm}
\usepackage{enumitem}
\usepackage{latexsym}
\usepackage{graphicx, epsfig}
\usepackage{amssymb,amsmath}
\usepackage{multimedia}
\usepackage{subfigure}
\usepackage{cases}
\usepackage{mathrsfs}
\usepackage{blkarray}
\usepackage{pgfpages,tikz,tikz-cd,pgfkeys,pgfplots}
\usetikzlibrary{arrows,positioning,matrix,fit,backgrounds,shapes}

\usetikzlibrary{fit,matrix}
\tikzset{
mN/.style = {
    draw=#1, semithick, inner sep=0pt}
             }

\newtheorem{theorem}{Theorem}[section]

\newtheorem{lemma}[theorem]{Lemma}
\newtheorem{proposition}[theorem]{Proposition}
\newtheorem{corollary}[theorem]{Corollary}

\theoremstyle{definition}

\title{Row graphs of Toeplitz matrices\thanks{This work was partially supported by Science Research Center Program through the National Research Foundation of Korea(NRF) Grant funded by the Korean Government (MSIP)(NRF-2016R1A5A1008055). G.-S. Cheon was partially supported by the NRF-2019R1A2C1007518. Bumtle Kang was partially supported by the NRF-2021R1C1C2014187. S.-R. Kim was partially supported by the Korea government (MSIP) (NRF-2017R1E1A1A03070489).}}
\date{}

\author{ Gi-Sang Cheon$^{a, b}$, Bumtle Kang$^{b}$, Suh-Ryung Kim$^{b, c}$, and Homoon Ryu$^{b,c}$  \\
{\footnotesize $^a$ \textit{Department of Mathematics, Sungkyunkwan
University, Suwon 16419, Rep. of Korea}}\\
{\footnotesize $^{b}$ \textit{Applied Algebra and Optimization
Research Center, Sungkyunkwan University,}}\\{\footnotesize\textit{
Suwon 16419, Rep. of Korea}}\\
 {\footnotesize $^{c}$ \textit{Department of Mathematics Education,
Seoul National University,}}\\{\footnotesize\textit{
Seoul 08826, Rep. of Korea}}\\
{\footnotesize gscheon@skku.edu, lokbt1@skku.edu, srkim@snu.ac.kr, and ryuhomun@naver.com}}

\begin{document}

\maketitle
\begin{abstract}
In this paper, we study row graphs of Toeplitz matrices.
The notion of row graphs was introduced by Greenberg~{\em et al.} in 1984 and is closely related to the notion of competition graphs, which has  been extensively studied since Cohen had introduced it in 1968.

To understand the structure of the row graphs of Toeplitz matrices, which seem to be quite complicated, we have begun with Toeplitz matrices whose row graphs are triangle-free.
We could show that if the row graph $G$ of a Toeplitz matrix $T$ is triangle-free, then $T$ has the maximum row sum at most $2$.
Furthermore, it turns out that $G$ is a disjoint union of paths and cycles whose lengths cannot vary that much in such a case.
Then we study $(0,1)$-Toeplitz matrices whose row graphs have only path components, only cycle components, and a cycle component of specific length, respectively.
In particular, we completely characterize a $(0,1)$-Toeplitz matrix whose row graph is a cycle.
\end{abstract}

\section{Introduction}
\label{sec:intro}
The {\em support} of a matrix $A$ is the $(0,1)$-matrix with the $(i,j)$-entry equal to $1$ if the $(i,j)$-entry of $A$ is nonzero, and equal to $0$, otherwise. The {\em digraph} of a square matrix is the digraph whose adjacency matrix is the support of the matrix
If a square matrix $A$ is symmetric, then the arc set of its digraph is symmetric and we consider the {\em graph} of $A$, instead of the digraph of $A$, which is obtained from the digraph of $A$ by replacing each directed cycle of length $2$ with an edge.

A {\em Toeplitz matrix} of order $n$ is an $n\times n$ matrix, $T=(t_{ij})$ where $t_{ij}=a_{i-j}$ for all $i,j=1, \ldots, n$, {\it i.e.} a matrix of the form
\begin{eqnarray}\label{Toeplitz}
T= \left(\begin{array}{cccc}
a_0 & b_1 & \cdots & b_{n-1} \\
a_1 & a_0  & \ddots &\vdots \\
\vdots & \ddots & \ddots  & b_1  \\
a_{n-1}&\cdots &a_1 & a_0
\end{array}\right),\quad b_i=a_{-i}.
\end{eqnarray}

Toeplitz matrices have been investigated by many authors with wide application in various fields of pure and applied mathematics.
For example, the matrices often appear in discrete differential or integral equations, physical data-processing, orthogonal polynomials, stationary processes, and moment problems, see  \cite{Goh,HR,Ioh} and references there in. Recently, the primitivity and the exponents of $(0,1)$-Toeplitz matrices have been studied in \cite{jkkc}.

In this paper, we study the ``row graphs" of $(0,1)$-Toeplitz matrices and obtain some interesting results.
Given a $(0,1)$ boolean matrix $A$, a graph $G$ is called the {\em row graph} of $A$  and denoted by ${\rm RG}(A)$ if the rows of $A$ are the vertices of $G$, and two vertices are adjacent in $G$ if and only if their corresponding rows have a nonzero entry in the same column of $A$.
This notion was studied by Greenberg~\cite{greenberg1984inverting}.
As a matter of fact, ${\rm RG}(A)$ is the ``competition graph" of the digraph of $A$, which was introduced by Cohen \cite{cohen1968interval} in 1968 associated with the study of ecosystems.
Given a digraph $D$, the {\it competition graph} of $D$ has $V(D)$ as its vertex set  and an edge $uv$ if and only if there are arcs $u \to w$ and $v \to w$ for some vertex $w$ in $D$ (see Figure~\ref{fig:Toeplitz1} for an illustration).
The notion of competition graphs and its variants have been extensively studied since Cohen introduced it.

 Most digraphs dealt in studying competition graphs have none of loops, multiple edges. Equivalently, the adjacency matrices of those digraphs are $(0,1)$-matrices with the zero main diagonal and not both $(i,j)$-entry and $(j,i)$-entry $1$. In this vein, we technically consider the family ${\mathcal T}$ of $(0,1)$-Toeplitz matrices $T=(t_{i,j})_n$ satisfying the following conditions:
 \begin{itemize}
 \item[(i)] $t_{i,i}=0$ for all $i \in [n]$;
 \item[(ii)] If $t_{i,j}=1$, then $t_{j,i}=0$.
 \end{itemize}
 If we describe the row graph $G$ of a $(0,1)$-Toeplitz matrix $(t_{i,j})_n$ more precisely,
 \[V(G)=[n]\]
 and $i$ and $j$ are adjacent in $G$ if and only if $t_{i,k}=t_{j,k}=1$ for some $k \in [n]$.

 If a graph $G$ is the row graph of $T$ for some $T \in \mathcal{T}$, then we call $G$ a {\em Toeplitz-row graph}. Unless otherwise mentioned, we mean by a $(0,1)$-Toeplitz matrix a matrix in ${\mathcal T}$.

 Given a $(0,1)$ boolean matrix $A$, we denote by ${\mathcal C}(A)$ the adjacency matrix of ${\rm RG}(A)$, that is, the $(i,j)$-entry of $\mathcal{C}(A)$ equals
\[\begin{cases}
	1 & \mbox{if $i \neq j$ and the inner product of the $i$th row and the $j$th row of $A$ is $1$;} \\ 0 & \mbox{otherwise.}
\end{cases}\]
 We may regard ${\mathcal C}$ as an operator on the set of all $n \times n$ boolean matrices.

 \begin{figure}
\begin{tikzpicture}[>=stealth,thick,baseline]
\node (matrix) [matrix of math nodes,left delimiter=(,right delimiter={)}]
{0 & 0 & 1 & 0 & 0 & 1 \\
1 & 0 & 0 & 1 & 0 & 0 \\
0 & 1 & 0 & 0 & 1 & 0 \\
1 & 0 & 1 & 0 & 0 & 1 \\
0 & 1 & 0 & 1 & 0 & 0 \\
0 & 0 & 1 & 0 & 1 & 0 \\};
\node [left=3mm of matrix]{$T=$};
\end{tikzpicture}
\qquad
\begin{tikzpicture}[baseline={(0,2.2)}]
\filldraw (0,4.5) circle (2pt);
\draw (0,4.8) node{\small $1$};
\filldraw (-1.3,3.85) circle (2pt);
\draw (-1.6,3.85) node{\small $6$};
\filldraw (1.3,3.85) circle (2pt);
\draw (1.6,3.85) node{\small $2$};
\filldraw (-1.3,2.55) circle (2pt);
\draw (-1.6,2.55) node{\small $5$};
\filldraw (1.3,2.55) circle (2pt);
\draw (1.6,2.55) node{\small $3$};
\filldraw (0,1.5) circle (2pt);
\draw (0,1.2) node{\small $4$};

\draw (0,1.5) -- (1.3,3.85);

\draw (1.3,2.55) -- (-1.3,2.55);

\draw (0,1.5) -- (-1.3,3.85);
\draw (0,4.5) -- (-1.3,3.85);
\draw (0,1.5) -- (0,4.5);

\draw (1.3,3.85) -- (-1.3,2.55);

\draw (1.3,2.55) -- (-1.3,3.85);

\draw (0,1.5) -- (0,4.5);

\draw (0,0.75) node{${\rm RG}(T)=:G$};
\end{tikzpicture}
\qquad
\begin{tikzpicture}[>=stealth,thick,baseline]
\node (matrix) [matrix of math nodes,left delimiter=(,right delimiter={)}]
{0 & 0 & 0 & 1 & 0 & 1 \\
0 & 0 & 0 & 1 & 1 & 0 \\
0 & 0 & 0 & 0 & 1 & 1 \\
1 & 1 & 0 & 0 & 0 & 1 \\
0 & 1 & 1 & 0 & 0 & 0 \\
1 & 0 & 1 & 1 & 0 & 0\\};
\node [left=3mm of matrix]{$A(G)=$};
\end{tikzpicture}

\caption{A Toeplitz matrix $T=T_6 \langle 1,3;2,5\rangle$, the row graph $G$ of $T$, and the adjacency matrix $A(G)$ of $G$}
\label{fig:Toeplitz1}
\end{figure}
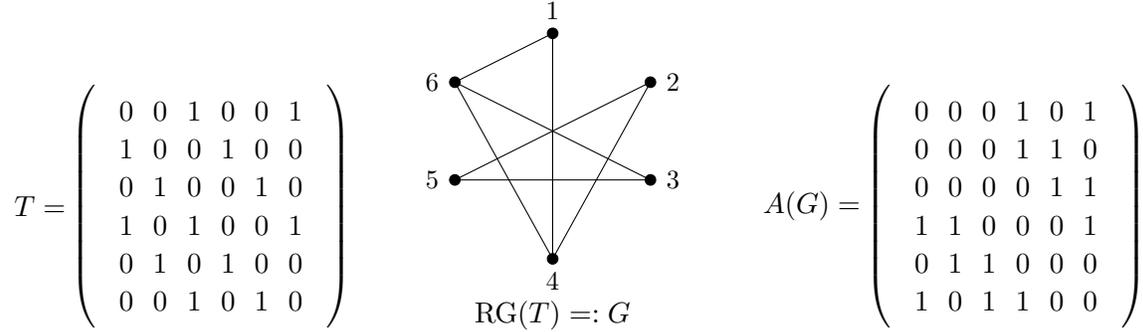
In Section~\ref{sec:TCgraphs}, we study triangle-free Toeplitz-row graphs.
 We first derive equalities involving integers belonging to $\alpha$ or $\beta$ as the condition for two vertices being adjacent in the row graph of a Toeplitz matrix $T_n\langle \alpha;\beta\rangle$ (Theorem~\ref{thm:edgecond}).
We show that if the row graph of a Toeplitz matrix $T$ is triangle-free, then $T$ belongs to ${\mathcal T}_{\le 2}$ where ${\mathcal T}_{\le 2}$ denotes the set of Toeplitz matrices with the maximum row sum at most $2$.
By the way, if $T \in {\mathcal T}_{\le 2}$, then ${\rm RG}(T)$ is a disjoint union of cycles and paths and therefore a triangle-free Toeplitz row graph is a disjoint union of paths and cycles.
Based on these observations we came up with the following research question:
\begin{center} Is every disjoint union of cycles and paths a triangle-free Toeplitz-row graph?
\end{center}	
The answer turns out to be no.
As a matter of fact, the lengths of paths or cycles in a triangle-free Toeplitz row graph cannot be much different (Theorem~\ref{thm:forbidden}).
 We give a necessary and sufficient condition for a Toeplitz matrix belonging to ${\mathcal T}_{\le2}$ (Theorem~\ref{thm:ind2}), which plays an important role throughout Sections~\ref{sec:TCgraphs}.
 We characterize a Toeplitz matrix $T_n(\alpha,\beta) \in {\mathcal T}_{\le 2}$ with $|\alpha|=|\beta|=2$ whose row graph is triangle-free (Theorem~\ref{thm:chartriangle-free}) and  the connected triangle-free Toeplitz-row graphs (Theorem~\ref{thm:connected triangle-free}).

 Section~\ref{sec:rowsum2} focuses on $(0,1)$-Toeplitz matrices whose row graphs consist of only path components, only cycle components, and a cycle component with specific length, respectively.
 We first give a sufficient condition for a $(0,1)$-Toeplitz matrix to have the row graph consisting of only path components (Proposition~\ref{prop:acylcic}).
 Then we show that, given integers $4\le m\le n$ except a few cases,  there is a Toeplitz matrix of order $n$ whose row graph contains a cycle component of length $m$ (Theorem~\ref{thm:cyclemn}).
 We completely characterize the Toeplitz matrices $T_n \langle \alpha;\beta \rangle$ with $|\alpha|=2$ and $|\beta|=1$ whose row graph contains a cycle (Theorem~\ref{thm:s1t2cycle}).
 We give a sufficient condition for the row graph of a Toeplitz matrix consisting of only cycle components and figure out the length of cycles
 (Proposition~\ref{prop:cdistoep} and Theorem~\ref{thm:dcycle}).
 Finally, we completely characterize a Toeplitz matrix whose row graph is a cycle (Theorem~\ref{thm:cycle}).

\section{Triangle-free Toeplitz-row graphs}
\label{sec:TCgraphs}
Since Toeplitz matrices are completely determined by their first column and first row, we begin by assuming that there are $k_1$ ones in the first column and $k_2$ ones in the first row of $T$ ($k_1$ or $k_2$ may be $0$ but not both), and $t_{i+1,1}=1$ for $i=i_1,\ldots,i_{k_1}$ if $k_1 \ge 1$ and $t_{1,j+1}=1$ for $j=j_1,\ldots,j_{k_2}$ if $k_2 \ge 1$.
Denote by $\alpha=\{i_1,\ldots,i_{k_1}\}$ if $k_1 \ge 1$ and $\beta=\{j_1,\ldots,j_{k_2}\}$ if $k_2 \ge 1$  the sets of such row and column indices in increasing order.
 If $k_1=0$ (resp.\ $k_2=0$), then $\alpha=\emptyset$ (resp.\ $\beta=\emptyset$).
 By definition, clearly $\alpha$ and $\beta$ are disjoint subsets of $[n-1]$. Recall that the lower shift matrix $L$ of order $n$ is the $(0,1)$-matrix with ones only on the subdiagonal, and zeroes elsewhere. For example, the $4\times 4$ lower shift matrix is
\begin{eqnarray*}
L= \left(\begin{array}{cccc}
0 & 0 & 0 &0 \\
1 &0  &0 &0 \\
0&1&0 & 0  \\
0&0 &1& 0
\end{array}\right).
\end{eqnarray*}
Thus it follows from (\ref{Toeplitz}) that $T$ can be expressed in terms of $L$ and $U=L^T$ as follows:
\begin{equation}
 T = \sum_{m=1}^{k_1} L^{i_m} + \sum_{m=1}^{k_2} U^{j_m}
 \label{eq:toedecom}
 \end{equation}
 (if $k_1=0$ or $k_2=0$, then the corresponding summands are regarded as zero matrices).
 For convenience, we denote $T$ by $T_n \langle \alpha;\beta\rangle$ or $T_n \langle i_1,\ldots,i_{k_1};j_1,\ldots,j_{k_2}\rangle$ when we need to specify the elements of $\alpha,\beta$.
 Especially,  if $\alpha=\beta$, then $T_n \langle \alpha;\beta\rangle$ is a $(0,1)$-symmetric Toeplitz matrix with zero diagonal and we denote it by $ST_n \langle \alpha \rangle$.

We note that $T_n \langle \alpha;\beta\rangle=PT_n \langle \beta;\alpha\rangle P^{T}$ for the permutation matrix $P=[p_{i,j}]$ where $p_{i,j}=1$ if and only if $i+j=n+1$.
Thus the row graph of $T_n \langle \alpha;\beta\rangle$ and that of $T_n \langle \beta;\alpha\rangle$ are isomorphic.
Accordingly, we may assume $i_1<j_1$ for a $(0,1)$-Toeplitz matrix $T_n\langle i_1,\ldots, i_{k_1};j_1,\ldots,j_{k_2}\rangle$ without loss of generality.

\begin{theorem}\label{thm:edgecond}
Let $T=T_n \langle \alpha;\beta\rangle$ and take integers $i$ and $j$ in $[n]$ with $i<j$.
Then $i$ and $j$ are adjacent in ${\rm RG}(T)$ if and only if one of the following is true:
\begin{numcases}
{j-i=}
i_t-i_s &  for some  $i_s \in \alpha \cap [i-1]$ and  $i_t \in \alpha \cap [j-1]$ with $s < t$;  \label{eq:case1}\\
j_q-j_p & for some $j_q \in \beta \cap [n-i]$ and  $j_p \in \beta \cap  [n-j]$ with $p < q$;\label{eq:case3} \\
i_t+j_q & for some $i_t \in \alpha$ and $j_q \in \beta$.  \label{eq:case2}
\end{numcases}
\end{theorem}
\begin{proof} Let $T=[t_{i,j}]$ and $G={\rm RG}(T)$.
Suppose that $i$ and $j$ are adjacent in $G$.
Then, by definition, there exists  $k\in[n]$ such that $t_{i,k}=1$ and $t_{j,k}=1$.
We consider the following three cases: (i) $1\le k<i$; (ii) $i<k<j$; (iii) $j<k\le n$.
We let
\[i-k=i_s, \quad j-k=i_t, \quad k-i=j_q, \quad \mbox{and} \quad k-j=j_p. \]
In Case (i), $i_s\in \alpha \cap [i-1]$, $ i_t \in \alpha \cap [j-1]$, and $j-i=i_t-i_s$;
in Case (ii), $j_q \in \beta$, $i_t \in \alpha$, and $j-i=i_t+j_p$;
in Case (iii), $j_q \in  \beta \cap [n-i]$, $j_p \in \beta \cap [n-j]$, and $j-i=j_p-j_q$.

We show the converse.

{\it Case 1.} $j-i=i_t-i_s$ for some  $i_s \in \alpha \cap [i-1]$ and  $i_t \in \alpha \cap [j-1]$.
Then  $t_{i_s+1,1}=t_{i_t+1,1}=1$.
Since $i_s \le i-1$, we have $i-i_s \ge 1$ and $t_{i,i-i_s}=1$.
Similarly, $t_{j,j-i_t}=1$.
By the case assumption, $i-i_s=j-i_t$, so $i$ and $j$ are adjacent in $G$.

{\it Case 2.} $j-i=j_q-j_p$ for some $j_q \in \beta \cap [n-i]$ and  $j_p \in \beta \cap [n-j]$.
Then $t_{1,j_p+1}=t_{1,j_q+1}=1$.
Since $j_q \le n-i$, we have $i+j_q \le n$ and $t_{i,i+j_q}=1$.
Similarly, $t_{j,j+j_p}=1$.
By the case assumption, $i+j_q=j+j_p$, so $i$ and $j$ are adjacent in $G$.

{\it Case 3.} $j-i=i_t+j_q$ for some $i_t \in \alpha$ and $j_q \in \beta$.
Then $t_{1,j_q+1}=t_{i_t+1,1}=1$.
Therefore $t_{i,i+j_q}=t_{j,j-i_s}=1$.
By the case assumption, $i+j_q=j-i_s$, so $i$ and $j$ are adjacent in $G$.
 \end{proof}

\begin{proposition}
For any Toeplitz matrix of order $n$, the $\ell$th row sum equals the $(n-\ell+1)$st column sum for any $\ell \in [n]$.
\end{proposition}
\begin{proof}
Let $T=T_n\langle \alpha;\beta \rangle$.
Take $\ell \in [n]$ and let\[
M_1 = \{ \ell-i_s \in [n] \mid  i_s \in \alpha\}\quad \mbox{and} \quad M_2 = \{ \ell+j_{t}\in [n] \mid j_t \in \beta\}.\]
Then the nonzero columns in the $\ell$th row is a disjoint union of $M_1$ and $M_2$. Now we let\[
N_1 = \{ p \in [n] \mid p -  i_{s} = n-\ell+1 , i_s \in \alpha\}\quad \mbox{and} \quad N_2 = \{ p\in[n] \mid p+j_{t}=n-\ell+1, j_t \in \beta\}.\]
Then the nonzero entries on $n-\ell+1$st column is a disjoint union of $N_1$ and $N_2$.

We claim that for each $i=1,2$ and $p \in [n]$, $x \in M_{i}$ if and only if $n-p+1 \in N_{i}$ as follows:
\begin{tabbing}
\= $p \in M_1$ (resp.\ $p \in M_2$) \\
\> $\Leftrightarrow$ $p = \ell - i_{s}$ for some $i_s \in \alpha$ (resp.\ $p=\ell+j_t$ for some $j_t \in \beta$)\\
\> $\Leftrightarrow$ $n-p+1 = n-\ell+1 +i_s$ for some $i_s \in \alpha$ (resp.\  $n-p+1 = n-\ell+1 - j_t$ for some $j_t \in \beta$ )\\
\> $\Leftrightarrow$ $(n-p+1)-i_s=n-\ell+1$ for some $i_s \in \alpha$ (resp.\  $(n-p+1)+j_t = n-\ell+1$ for some $j_t \in \beta$ )\\
\> $\Leftrightarrow$ $n-p+1 \in N_1$ (resp.\ $n-p+1 \in N_2$).
\end{tabbing}
Thus $|M_1 \cup M_2|= |N_1 \cup N_2|$ and this completes the proof.
\end{proof}

\begin{corollary}\label{cor:maxinout} The maximum row sum and the maximum column sum of a Toeplitz matrix are the same. \label{thm:degsym}\end{corollary}

\begin{corollary}\label{cor:trinaglefree}
If the row graph of a Toeplitz matrix $T$ is triangle-free, then the maximum row sum of $T$ is at most two.	
\end{corollary}
\begin{proof}
Suppose that a triangle-free graph $G$ is the row graph of a Toeplitz matrix $T$.
Then, by the definition of row graph, the maximum column sum is at most two.
Thus, by Corollary~\ref{cor:maxinout}, the maximum row sum of $T$ is at most two.
\end{proof}

In \cite{ijcomp}, it is shown that if $D$ has maximum indegree and maximum outdegree at most $2$, then its competition graph is a disjoint union of cycles and paths.
In other words, if a Toeplitz matrix $T$ belongs to ${\mathcal T}_{\le 2}$, then ${\rm RG}(T)$ is a disjoint union of cycles and paths.

\begin{corollary}\label{cor:starcomp} If $T$ is a Toeplitz matrix with the maximum row sum at most $2$, then the row graph of $T$ is a disjoint union of cycles and paths.
\end{corollary}

Thus, by Corollaries~\ref{cor:trinaglefree} and \ref{cor:starcomp}, we obtain the following corollary.
\begin{corollary}\label{cor:starcomp1} A triangle-free Toeplitz-row graph is a disjoint union of cycles and paths.
\end{corollary}

Let ${\mathcal G}_{\le 2}$ be the set of graphs with degree at most $2$, that is, the set of graphs each of which is a disjoint union of cycles and paths. Then Corollary~\ref{cor:starcomp1} may be restated as follows:
\[\{G \mid G={\rm RG}(T) \text{ for some $T \in {\mathcal T}$ and } G \text{ is triangle-free}\} \subseteq {\mathcal G}_{\le 2}.\]

In the rest of this section, we will examine disjoint union of cycles and paths that are row graphs of Toeplitz matrices.
By Corollaries~\ref{cor:trinaglefree} and \ref{cor:starcomp}, we only need take a look at Toeplitz matrices in ${\mathcal T}_{\le 2}$.
We can easily check that if $T_n \langle \alpha ; \beta \rangle \in {\mathcal T}_{\le 2}$, then $\alpha \le 2$ and $\beta \le 2$.
However, the following matrix shows that the converse is not true;
\begin{align*}
T_7 \langle 1,3;2,5 \rangle= \left(\begin{array}{ccccccc}
0 & 0 & 1 & 0 & 0 & 1 & 0\\
1 & 0 & 0 & 1 & 0 & 0 & 1\\
0 & 1 & 0 & 0 & 1 & 0 & 0\\
1 & 0 & 1 & 0 & 0 & 1 & 0\\
0 & 1 & 0 & 1 & 0 & 0 & 1\\
0 & 0 & 1 & 0 & 1 & 0 & 0\\
0 & 0 & 0 & 1 & 0 & 1 & 0
\end{array}\right).
\end{align*}
 Based on the observation,  we present a necessary and sufficient condition for a Toeplitz matrix $T=T_n \langle \alpha ; \beta \rangle$
belonging to ${\mathcal T}_{\le 2}$ when $\alpha \ne \emptyset$ and $\beta\ne \emptyset$, which will play a key role in studying triangle-free Toeplitz-row graphs.
We can easily check that  $T=T_n \langle \alpha ; \beta \rangle$ belongs to ${\mathcal T}_{\le 2}$  if one of $\alpha$ and $\beta$ is the empty set and the other set has order at most $2$.

\begin{theorem} \label{thm:triangle-free}
Let $T=T_{n}\langle i_1,\ldots,i_{k_1};j_1,\ldots,j_{k_2}\rangle$ where $k_1,k_2 \ge 1$.
Then the maximum row sum of $T$ is at most $2$ if and only if the following property holds:
\begin{itemize}
\item[$(\star)$] $k_1, k_2 \le 2$, and especially if $k_2 = 2$, then $i_1 + j_2 \ge n$ and if $k_1=2$, then  $i_2 + j_1 \ge n$.
\end{itemize}
\label{thm:ind2}
\end{theorem}
\begin{proof}
We first show the `if' part. To reach a contradiction, suppose that the row sum of the $u$th row is at least $3$ for some $u \in [n]$. Then there are $\{v_1,v_2,v_3\} \subset [n]$ such that the entries $(u,v_1)$, $(u,v_2)$, and $(u,v_3)$ of $T$ is $1$. By the definition of Toeplitz matrix, for each $i=1,2,3$, either\begin{equation}
u-v_i \in \alpha\label{eq:1} \end{equation}
or \begin{equation}
v_i-u \in \beta.\label{eq:2}\end{equation}

Since $k_1,k_2\le 2$, exactly two of $v_1,v_2,v_3$ satisfy \eqref{eq:1} or exactly two of $v_1,v_2,v_3$ satisfy \eqref{eq:2}.
We may assume the former so that $u-v_1 = i_1$, $u-v_2 =i_2$, and $v_3-u \in \beta$.
Then $v_3- v_2 \ge i_2 + j_1 \ge n$, which is impossible since $\{v_2,v_3\} \subseteq [n]$.
Thus the maximum row sum of $T$ is at most $2$.

Now we show the contrapositive of the `only if' part.
We can easily see that the $n$th row sum and the $1$st row sum are $k_1$ and $k_2$, respectively. Therefore $k_1, k_2 \le 2$.
If $k_2 = 2$ and $i_1+j_2 < n$, then the entries $(n-j_2, n-i_1-j_2)$, $(n-j_2,n+j_1-j_2)$, and $(n-j_2,n)$ equal $1$ and so the $(n-j_2)$th row sum is at least $3$.
A similar argument can be applied for the case $k_1 = 2$ and $i_2+ j_1 < n$.
\end{proof}

We characterize a Toeplitz matrix $T_n(\alpha,\beta) \in {\mathcal T}_{\le 2}$ with $|\alpha|=|\beta|=2$ whose row graph is triangle-free.
\begin{theorem} \label{thm:chartriangle-free}
Let $T=T_n\langle i_1,i_2 ; j_1,j_2 \rangle$ with $i_1+j_2 \ge n$ and $i_2 +j_1 \ge n$. Then ${\rm RG}(T)$ contains a triangle if and only if one of the following holds:
\begin{itemize}
\item[{\rm (i)}] $2i_2 = 3i_1+j_1$ and $n >2i_1+j_1$;
\item[{\rm (ii)}] $2j_2 = i_1+3j_1$ and $n > i_1+2j_1$;
\item[{\rm (iii)}] $i_2 = 3i_1+2j_1$;
\item[{\rm (iv)}] $j_2 = 2i_1+3j_1$;
\item[{\rm (v)}] $i_2+j_2 = 2(i_1+j_1)$.
\end{itemize}
\end{theorem}
\begin{proof} We first show the `if' part.
We can check that for each cases (i), (ii), (iii), (iv), and (v),
$\{1+i_1, 1+i_2, 1+ 2i_2-i_1\}$, $\{1, 1+j_2-j_1, 1+ 2j_2-2j_1\}$, $\{1+i_1,1+2i_1+j_1, 1+3i_1+2j_1\}$, $\{1,1+i_1+j_1, 1+2i_1+2j_1\}$, and $\{1,1+j_2-j_1, 1+i_2+j_2-(i_1+j_1)\}$
is a triangle of ${\rm RG}(T)$, respectively.

Now we show the `only if' part.
Suppose that there is a triangle $uvw$ of ${\rm RG}(T)$ with $u<v<w$.
Then $\{v-u,w-v,w-u\} \subset \{i_1+j_1,j_2-j_1,i_2-i_1\}$ by the hypothesis and Theorem~\ref{thm:edgecond}.
Suppose that $\{w-v,v-u\}$ is either $\{i_1+j_1,i_2-i_1\}$ or $\{i_1+j_1,j_2-j_1\}$.
Then $w-u=i_2+j_1 \ge n$ or $w-u=i_1+j_2 \ge n$, which is impossible.
Suppose that $\{w-v,v-u\} = \{i_2-i_1\}$ and $w-u = j_2-j_1$. Then, by Theorem~\ref{thm:edgecond}, $u > i_1$ and $w+j_1 \le n$.
However, $w+j_1 = u+j_2 > i_1+j_2 \ge n$ and we reach a contradiction. Similarly, we can show that the case $\{w-v,v-u\} = \{j_2-j_1\}$ and $w-u = i_2-i_1$ cannot happen either.

Then the following are the only possible remaining cases since $w-u = (w-v)+(v-u)$:
\begin{itemize}
\item[(1)] $\{w-v,v-u\}=\{i_2-i_1\}$ and $w-u = i_1+j_1$;
\item[(2)] $\{w-v,v-u\}=\{j_2-j_1\}$ and $w-u = i_1+j_1$;
\item[(3)] $\{w-v,v-u\}=\{i_1+j_1\}$ and $w-u = i_2-i_1$;
\item[(4)] $\{w-v,v-u\}=\{i_1+j_1\}$ and $w-u = j_2-j_1$;
\item[(5)] $\{w-v,v-u\}=\{i_2-i_1,j_2-j_1\}$ and $w-u = i_1+j_1$.
\end{itemize}
In the case (1), since $u>i_1$, $n \ge w =u+i_1+j_1 > 2i_1+j_1$. Moreover, since $w-u = (w-v)+(v-u)$, $2i_2 = 3i_1+j_1$.
In the case (2), since $w+j_1 \le n$, $n \ge w+j_1 = u+i_1+2j_1 > i_1+2j_1$. In addition, $2j_2 = i_1+3j_1$.
In the case (3), since $u>i_1$, $n \ge w =u+2(i_1+j_1) > 3i_1+2j_1$.
In the case (4), since $w+j_1 \le n$, $n \ge w+j_1 \ge u+2(i_1+j_1)+j_1 > 2i_1+3j_1$.

Furthermore, in the case (3), $i_2 = 3i_1+2j_1$ and $n > 3i_1+2j_1 = i_2$; in the case (4) $j_2 = 2i_1+3j_1$ and $n > 2i_1+3j_1 = j_2$. By the definition of Toeplitz graph, $n > \max\{i_2,j_2\}$, so there is no restriction on $n$ in cases (3) and (4).
Finally, we have  $i_2+j_2 = 2(i_1+j_1)$ in the case (5). Hence we obtain the desired result.
\end{proof}

To find out the exact number of cycles or paths, we need the following lemma.
	\begin{lemma}\label{lemma:path-1}
		Take $T=T_n\langle i_1,\ldots,i_{k_1};j_1,\ldots,j_{k_2} \rangle \in {\mathcal T}_{\le 2}$ with $k_1,k_2 \ge 1$ and let $G$ be the row graph of $T$. Then the following are true:
		\begin{itemize}
			\item[(a)] If $v_1v_2\cdots v_\ell$ is a path in $G$ for an integer $\ell \ge 2$,
			then $(v_1-1)(v_2-1)\cdots (v_\ell -1)$ is also a path in $G$ if $v_i \notin \{1, 1+i_1\}$ for any $i\in [\ell]$.
			\item[(b)] If $v_1v_2\cdots v_\ell$ is a path component in $G$ for an integer $\ell \ge 2$ with $v_i \notin \{1, 1+i_1\}$ for each $i\in [\ell]$ and
			\begin{equation}\label{eq:forbidden}
\{v_1,v_\ell\} \cap \{n-i_2+i_1+1, n-j_2+1, n-j_1+1, n-(i_1+j_1)+1\}=\emptyset,
\end{equation} then $(v_1-1)(v_2-1)\cdots (v_\ell -1)$ is also a path component in $G$ (if $|\alpha|= 1$, then we ignore $i_2$ and if $|\beta|=1$, then we ignore $j_2$).
			\item[(c)] If $v_1v_2\cdots v_\ell v_1$ is a cycle in $G$ for an integer $\ell \ge 2$, then $(v_1-1)(v_2-1)\cdots (v_\ell -1)(v_1-1)$ is also a cycle in $G$ if $v_i \notin \{1,1+i_1\}$ for any $i\in [\ell]$.
		\end{itemize}
	\end{lemma}
	\begin{proof}
	By Theorems~\ref{thm:edgecond} and \ref{thm:ind2}, vertices $x$ and $y$ with $y>x$ are adjacent if and only if one of the following is true:
\begin{itemize}
\item[(i)] $y-x = i_2-i_1$, $i_1 \le x-1$, and $i_2 \le y-1$;
\item[(ii)] $y-x = j_2-j_1$, $j_2 \le n-x$, and $j_1 \le n-y$;
\item[(iii)] $y-x=i_1+j_1$
\end{itemize}
(if $k_1= 1$ (resp. $k_2 =1$), then we ignore the term $i_2-i_1$ (resp. $j_2-j_1$)).

Suppose that $v_1v_2 \cdots v_\ell$ is a path in $G$.
To show (a), suppose that $S \cap \{1,1+i_1\}=\emptyset$.
Fix $k \in [\ell-1]$ and $x_k = \min\{v_{k+1},v_k\}$ and $y_k = \max\{v_{k+1},v_k\}$.
Since $x_k\neq 1$ by the hypothesis, $1 \le x_k-1 \le y_k-1$.
Since $x_k$ and $y_k$ are adjacent, one of (i), (ii), (iii) for $x=x_k$ and $y=y_k$ is true.

In the case (i), $i_1+1 \le x_k-1$ since $x_k \neq 1+i_1$.
Thus $i_1 \le (x_k-1)-1$.
Since $i_1-x_k=i_2-y_k$, $i_2 \le (y_k-1)-1$.
In the case (ii), $j_2 < n-(x_k-1)$ and $j_1 < n-(y_k-1)$.
Thus $x_k-1$ and $y_k-1$ are adjacent in $G$.
Since $k$ was arbitrarily chosen, we may conclude that $v_i-1$ and $v_{i+1}-1$ are adjacent in $G$ for all $i \in [\ell-1]$.
Consequently, we have shown that $(v_1-1)(v_2-1)\cdots (v_\ell -1)$ is also a path in $G$.
	
To show (b), suppose $v_1v_2 \cdots v_\ell$ is a path component in $G$ for an integer $\ell \ge 2$ with $v_i \notin \{1, 1+i_1\}$ for any $i\in [\ell]$ and satisfying \eqref{eq:forbidden}.
		Then $Q:=(v_1-1)(v_2-1) \cdots (v_\ell-1)$ is a path by (a).
		Suppose, to the contrary, $(v_1-1)(v_2-1)\cdots(v_\ell -1)$ is not a path component. Then $Q$ is extendded to a longer path by Theorem~\ref{thm:ind2}.
		We consider the case in which there is a vertex $v\neq v_{\ell-1}-1$ which is adjacent to  $v_\ell -1$ in $G$.	
		Let
		\[x=\min\{v,v_\ell-1\} \quad \mbox{and} \quad  y=\max\{v,v_\ell-1\}.\]
	    Since $x$ and $y$ are adjacent, one of (i), (ii), (iii) is true.	
		
		Suppose $v+1 > n$.
		Since $v$ is a vertex in $G$, $v=n=y$.
		Then either (i) or (iii) is true.		
		Thus either $v_\ell-1 = n-(i_1+j_1)$
		or $v_\ell-1 = n-i_2+i_1$, which contradicts \eqref{eq:forbidden}.
		
		Now, suppose that $v+1 \le n$.
		Noting $y-x=(y+1)-(x+1)$, we know that $v+1$ and $v_\ell$ are adjacent in the case (iii).
		In the case (ii),
		$y-x=i_2-i_1$, $i_1 \le x-1 <x$, and $i_2 \le y-1<y$. Thus $v+1$ and $v_\ell$ are adjacent in $G$ and we reach a contradiction.
		In the case (i), $y-x = j_2-j_1$, $j_2 \le n-x$, and $j_1 \le n-y$.
		Since $v_\ell-1 \not\in \{n-j_1, n-j_2\}$ by \eqref{eq:forbidden}, $j_2 \le n-(x+1)$ or  $j_1 \le n-(y+1)$.
		Since $y-x = j_2-j_1$, $j_2 \le n-(x+1)$ and  $j_1 \le n-(y+1)$.
		Thus $v+1$ and $v_\ell$ are adjacent in $G$ and we reach a contradiction.

		By a similar argument, one can show that there is no vertex except $v_2-1$ that is adjacent to $v_1 -1$ in $G$.
		Hence, $(v_1-1)(v_2-1)\cdots(v_\ell-1)$ is a path component in $G$ and (b) is valid.
		
		To show (c), suppose $v_1v_2\cdots v_\ell v_1$ is a cycle in $G$.
		By (a), $(v_1-1)(v_2-1)\cdots(v_\ell-1)$ and $(v_\ell-1)(v_1-1)$ are paths in $G$.
		Therefore $(v_1-1)(v_2-1)\cdots(v_\ell -1)(v_1-1)$ is also a cycle in $G$.
	\end{proof}

\begin{theorem}\label{thm:forbidden}
Let $G$ be a triangle-free Toeplitz-row graph.
Then the following are true:
\begin{itemize}
	\item[(a)] There are no three cycle of different lengths in $G$.
	\item[(b)] There are no seven paths of different lengths in $G$.
\end{itemize}
\end{theorem}

\begin{proof}
Since $G$ is triangle-free, $G$ is the row graph of a Toeplitz matrix $T_n \langle i_1,\ldots,i_{k_1}; j_1,\ldots,j_{k_2} \rangle$ satisfying ($\star$) by Theorem~\ref{thm:ind2}.

To show (a) by contradiction, suppose that there are cycles $C_1, C_2, C_3$ of different lengths in $G$.
Then, by Lemma~\ref{lemma:path-1}(c), there exist cycles $C_1^*,C_2^*,C_3^*$ in $G$ such that $C_i^*$ has same length as $C_i$ for each $i =1,2,3$ and each of $C_1^*, C_2^*, C_3^*$ contains a vertex $1$ or $1+j_1$. By the pigeonhole principle, a vertex $1$ or $1+i_1$ is on two of cycles, which is a contradiction as $C_1$, $C_2$, $C_3$ are components of $G$ by Corollary~\ref{cor:starcomp1}. Thus (a) is true.

To prove (b) by contradiction, suppose that there are path components $P_1, P_2, \ldots, P_7$ of different lengths in $G$.
Then there exist path components $P_1^*, P_2^*, \ldots, P_7^*$ in $G$ such that $P_i^*$ has same length as $P_i$ for each $i=1, \ldots, 7$ and each of $P_1^*,P_2^*\ldots, P_7^*$ contains one of the vertices $1, 1+i_1, n-i_2+i_1+1, n-j_2+1, n-j_1+1, n-i_1-j_1+1$ by Lemma~\ref{lemma:path-1}(b) (if $k_1 = 1$ (resp.\ $k_2=1$), ignore $n-j_2+1$ (resp.\ $n-i_2+i_1+1$)).
By the pigeonhole principle, at least one of the vertices $1, 1+i_1, n-i_2+i_1+1, n-j_2+1, n-j_1+1, n-i_1-j_1+1$ is in two path components which is a contradiction (if $k_1 = 1$ (resp.\ $k_2=1$), ignore $n-j_2+1$ (resp.\ $n-i_2+i_1+1$)).
\end{proof}

\begin{corollary}
Any graph containing three cycles of different lengths or seven path components of different lengths is not a Toeplitz-row graph.
\end{corollary}

Now we characterize triangle-free Toeplitz-row graphs without isolated vertices.
 \begin{lemma} \label{lem:short} Any Toeplitz-row graph of order $n \le 4$ has an isolated vertex.
\end{lemma}
\begin{proof}
Clearly, if $n=2$, then the row graph of a Toeplitz matrix is edgeless.
To the contrary, suppose that $G:={\rm RG}(T)$ has no isolated vertex for some $T=T_n \langle \alpha;\beta \rangle$ with $n \in \{3,4\}$.
Since the vertices $1$ and $n$ have degree at least one, $t_{1,k}=t_{n,\ell}=1$ for some $k, \ell \in [n]$ and so $n-\ell \in \alpha$ and $k-1 \in \beta$.
Thus $k_1 \ge 1$ and $k_2 \ge 1$.
We let $\alpha = \{i_1,\ldots,i_{k_1}\}$ and $\beta = \{j_1,\ldots,j_{k_2}\}$.
Then $k_1+k_2 \le n-1 \le 3$, so $k_1 =1$ or $k_2 = 1$.
We first consider the case $k_1 =1$.
Then $k_2 \le 2$ and no two adjacent vertices in $G$ satisfy  Theorem~\ref{thm:edgecond}\eqref{eq:case1} .
Since $G$ has no isolated vertex, the vertex $n$ is adjacent to some vertex $u$ for some $1\le u < n$.
Obviously, $u$ and $n$ do not satisfy  Theorem~\ref{thm:edgecond}\eqref{eq:case3}, so $n-u=i_1+j_t$ for some $t\in \{1,2\}$.
Since $u<n \le 4$, $n-u \le 3$ and so $i_1+j_t \le 3$.
Since we assumed $i_1 < j_1$ throughout this paper, $(i_1,j_1)=(1,2)$ and if there exists $j_2$, then $j_2=3$.
Then all the entries of $T$ are zero except
$t_{2,1}=\cdots = t_{n,n-1}=1$, $t_{1,3}=\cdots=t_{n,n-2}=1$, and $t_{1,4}=\cdots=t_{n,n-3}=0$ or $1$ depend upon the existence of $j_2$.
Thus $t_{3,j}=0$ except $j=2$.
Yet, $t_{i,2}=0$ except $i=3$, so the vertex $3$ is isolated in $G$ and we reach a contradiction.
Thus $k_1=2$ and $k_2 = 1$.
Then no two adjacent vertices satisfy  Theorem~\ref{thm:edgecond}\eqref{eq:case3}.
Since the vertex $1$ has degree at least one, $1$ and a vertex $v$ are adjacent for some $2 \le v\le n$.
Since $1$ and $v$ do not satisfy  Theorem~\ref{thm:edgecond}\eqref{eq:case1}, $v-1 = i_s+j_1$ for some $s \in \{1,2\}$.
Since we assume $i_1 < j_1$ throughout this paper, $i_s+j_1 \ge 3$, so $n=v=4$ and $i_s+j_1= 3$.
Then $k_1=1$ and $(i_1,j_1)=(1,2)$, which leads to a contradiction as above.
\end{proof}

\begin{proposition} \label{prop:cyclepath}
A path or a cycle of order $n$ is a Toeplitz-row graph if and only if $n \ge 5$. \label{prop:crpc}
\end{proposition}
\begin{proof}  The `only if' part immediately follows from Lemma~\ref{lem:short}.

To prove the `if' part, suppose $n \ge 5$.
 Let $T_1 = T_n \langle 1,2 ; n-2 \rangle$ (see Figure~\ref{fig:cycle} for an illustration).
Then, for each $i=2,\ldots, n-1$, the $(i,i-1)$-entry and the $(i+1,i-1)$-entry of $T_1$ are $1$, so the vertices $i$ and $i+1$ are adjacent in ${\rm RG}(T_1)$.
Moreover, the $(1,n-1)$-entry and $(n,n-1)$-entry are $1$, so the vertices $1$ and $n$ are adjacent.
Thus ${\rm RG}(T_1)$ contains a path of order $n$ and so, by Corollary~\ref{cor:starcomp1}, ${\rm RG}(T_1)$ is a path of order $n$.

To show a cycle of length $n$ is a Toeplitz-row graph, consider $T_2 = T_n \langle 1,2 ; n-2,n-1 \rangle$ (see Figure~\ref{fig:cycle}).
Then, by the above argument, $2,3,\ldots, n, 1$ is the sequence of a path in ${\rm RG}(T_2)$.
Since the $(1,n)$-entry and $(2,n)$-entry are $1$, the vertices $1$ and $2$ are adjacent.
Thus ${\rm RG}(T_2)$ contains a cycle of length $n$
and so, by Corollary~\ref{cor:starcomp1}, ${\rm RG}(T_2)$ is a cycle of length $n$.
\end{proof}

\begin{figure}
\begin{center}
\begin{minipage}[c]{.4\textwidth}
\begin{center}
\begin{align*}
T_1= \left(\begin{array}{cccccc}
0 & 0 & 0 & 0 & 1 & 0 \\
1 & 0 & 0 & 0 & 0 & 1 \\
1 & 1 & 0 & 0 & 0 & 0 \\
0 & 1 & 1 & 0 & 0 & 0 \\
0 & 0 & 1 & 1 & 0 & 0 \\
0 & 0 & 0 & 1 & 1 & 0
\end{array}\right)\end{align*}
\end{center}
\end{minipage}
\begin{minipage}[c]{.4\textwidth}
\begin{center}
\begin{align*}
T_2= \left(\begin{array}{cccccc}
0 & 0 & 0 & 0 & 1 & 1 \\
1 & 0 & 0 & 0 & 0 & 1 \\
1 & 1 & 0 & 0 & 0 & 0 \\
0 & 1 & 1 & 0 & 0 & 0 \\
0 & 0 & 1 & 1 & 0 & 0 \\
0 & 0 & 0 & 1 & 1 & 0
\end{array}\right)\end{align*}
\end{center}
\end{minipage}
\end{center}
\caption{Toeplitz matrices $T_1=T_6\langle 1,2;4 \rangle$ and $T_2=T_6 \langle 1,2 ; 4,5 \rangle$}\label{fig:cycle}

\end{figure}

Now we are ready to give a characterization of connected triangle-free Toeplitz-row graphs.

\begin{theorem}\label{thm:connected triangle-free}
A Toeplitz-row graph of order $n$ is connected and triangle-free if and only if $n \ge 5$ and it is a path or a cycle of order $n$. 	
\end{theorem}
\begin{proof}
The `if' part is valid by the `if' part of Proposition~\ref{prop:cyclepath}.
The `only if' part is true by Lemma~\ref{lem:short} and the `only if' of Proposition~\ref{prop:cyclepath}.  	
\end{proof}

 \section{Toeplitz matrices in ${\mathcal T}_{\le 2}$ whose row graph contains a cycle}
\label{sec:rowsum2}
In the previous section, we showed that
\[\{G \mid G={\rm RG}(T) \text{ for some $T \in {\mathcal T}$ and } G \text{ is triangle-free}\} \subseteq {\mathcal G}_{\le 2}.\]
In this section, we study
\[{\mathcal T_{\le 2}^{\star}}:=\{T \in {\mathcal T_{\le 2}} \mid {\rm RG}(T) \text{ contains a cycle}\}.\]
From Corollary~\ref{cor:starcomp}, we know that the row graph of a Toeplitz matrix $T \in {\mathcal T}_{\le 2}$ is a disjoint union of paths and cycles.
Therefore $T_{\le 2} \setminus {\mathcal T_{\le 2}^{\star}}$ is a set of Toeplitz matrix with maximum row sum at most $2$ whose row graph is a disjoint union of path components.

We first present a specialized version of Theorem~\ref{thm:edgecond} for a matrix in ${\mathcal T}_{\le 2}$.

\begin{lemma}\label{lem:smallend}
Let $T=T_n \langle i_1, i_2 ; j_1, i_2 \rangle$ be a Toeplitz matrix in ${\mathcal T}_{\le 2}$.
Then two vertices $u$ and $v$ with $u<v$ are adjacent in the row graph of $T$ if and only if one of the following is true:
\begin{itemize}
\item[{\rm (i)}] $1 \le u \le n-j_2$ and $v=u+j_2-j_1$;
\item[{\rm (ii)}] $i_1+1 \le u  \le n-i_2+i_1$ and $v=u+i_2-i_1$;
\item[{\rm (iii)}] $1\le u \le n-j_1-i_1$ and $v=u+i_1+j_1$.
    \end{itemize}
\end{lemma}
\begin{proof}
Suppose that $u$ and $v$ are adjacent.
Then, by Theorem~\ref{thm:edgecond}, $v-u \in \{j_2-j_1, i_2+i_1, i_1+j_1, i_1+j_2, i_2+j_1, i_2+j_2\}$.
If $v=u+j_2-j_1$, then $1 \le u \le n-j_2$ by \eqref{eq:case3} of the same theorem.
Suppose $v=u+i_2-i_1$.
Then $i_1 \le u-1$ by \eqref{eq:case1}, so $i_1+1 \le u$.
Moreover, $u+i_2-i_1 \le n$ and so $u \le n-i_2+i_1$.
If $v=u+i_1+j_1$, then $u+i_1+j_1 = v  \le n$ and so $1 \le u \le n-j_1-i_1$.
The converse is obviously true by Theorem~\ref{thm:edgecond}.
\end{proof}

\subsection{Toeplitz matrices in ${\mathcal T}_{\le 2}^{\star}$}

We first give a sufficient condition for a Toeplitz matrix belonging to $T_{\le 2} \setminus {\mathcal T_{\le 2}^{\star}}$.
A matrix is called {\em acyclic} if its digraph does not contain a directed cycle.
Using the following lemma, we show that $T \in T_{\le 2} \setminus {\mathcal T_{\le 2}^{\star}}$ if $T$ is acyclic.

 \begin{lemma}[\cite{jkkc}]\label{lem:cycle}  Let $T=T_n\langle i_1,\ldots,i_{k_1}; j_1,\ldots,j_{k_2}\rangle$. Then the digraph of $T$ contains a directed cycle $C_{s,t}$ of length $\frac{i_s+j_t}{\gcd(i_s,j_t)}$ for each $s \in [k_1]$ and $t \in [k_2]$ such that $i_s + j_t \le n$.\end{lemma}

\begin{proposition} \label{prop:acylcic}
Let $T$ be an acyclic Toeplitz matrix in ${\mathcal T}_{\le 2}$. Then the row graph of $T$ consists of only path components.
\end{proposition}
\begin{proof}
Let $G={\rm RG}(T)$.
By Theorem~\ref{thm:ind2}, $T=T_{n}\langle \alpha; \beta \rangle$ with $|\alpha| \le 2$ and $|\beta| \le 2$.
Moreover, since $T$ is acyclic, $i_1+j_1 >n$ by Lemma~\ref{lem:cycle} and so Lemma~\ref{lem:smallend}(iii)  cannot happen.
If $|\alpha|=|\beta|=1$, then none of (i) and (ii) of Lemma~\ref{lem:smallend} can happen and so $G$ is an empty graph.
Suppose that $|\alpha| =2$ or $|\beta|=2$.
Then, since (iii) of Lemma~\ref{lem:smallend} cannot happen, the edge set of $G$ equals $E_1$ if $|\alpha|=2$ and $|\beta|=1$; $E_2$ if $|\alpha|=1$ and $|\beta|=2$; $E_1 \cup E_2$ if $|\alpha|=|\beta|=2$ where\[
E_1 = \{\{i, i-i_1+i_2\} \mid 1 \le i \le n-i_2\} \mbox{ and } E_2 = \{\{j,j-j_1+j_2\}\mid 1+j_1 \le j \le n+j_1-j_2\}.\]
Suppose, to the contrary, that the subgraph $G[E_1]$ of $G$ induced by $E_1$ contains a cycle $C:=v_0v_1\ldots v_{\ell-1}v_0$ (we identify $v_{\ell}$ with $v_0$) for some integer $\ell \ge 3$.
Then  $|v_{i+1}-v_{i}| = i_2-i_1$ for all $i \in \{0,1,\ldots,\ell-1\}$.
Since $C$ is closed, there is $i \in \{0,1,\ldots,\ell-1\}$ such that $v_{i-1} < v_i$ and $v_i > v_{i+1}$. Then \[v_i-v_{i+1} = | v_{i+1}-v_i| = |v_i-v_{i-1}| = v_i - v_{i-1},\]
which implies $v_{i-1} = v_{i+1}$. Since $\ell \ge 3$, we have reached a contradiction.
By a similar argument, we may show that $G[E_2]$ does not contain a cycle.
Therefore the statement is true if either $|\alpha|=2$ and $|\beta|=1$ or $|\alpha|=1$ and $|\beta|=2$.

Suppose that $|\alpha|=|\beta|=2$.
Since $i_2+j_1 > i_1+j_1> n$, $E_1$ and $E_2$ are disjoint.
Moreover, the largest $n-i_1$ among the ends of the edges  in $E_1$ is less than the smallest $j_1$ among the ends of the edges in $E_2$.
Thus, if a cycle is contained in $G$, then it must be contained in $G[E_1]$ or in $G[E_2]$.
Since we have shown that $G[E_1]$ and $G[E_2]$ are disjoint unions of paths, the statement is valid.
\end{proof}

Now, it is interesting to find out which matrix in ${\mathcal T_{\le 2}^{\star}}$ has a its row graph containing a cycle component.
First of all, we will claim that for any integers $4 \le m \le n$ except a few cases, there exists a Toeplitz matrix of order $n$ in ${\mathcal T_{\le 2}^{\star}}$ whose row graph contains a cycle component of length $m$.

\begin{lemma}\label{lem:klcycle}
Let $T= T_n \langle i_1, i_2 ; j_1, j_2 \rangle$.
If there exists integers $1\le k \le \frac{n-1-i_1}{i_2-i_1}$ and $1 \le \ell \le \frac{n-1-j_1}{j_2-j_1}$ such that \[ i_1+j_1 = k(i_2-i_1)+ \ell(j_2-j_1),\]
then the row graph of $T$ contains a cycle of length $k+\ell+1$.
\end{lemma}
\begin{proof}
Let $s=i_2-i_1$ and $t=j_2-j_1$ and suppose that there exists integers $1\le k \le \frac{n-1-i_1}{s}$ and $1 \le \ell \le \frac{n-1-j_1}{t}$.
Then $i_1+ks \le n-1$ and $j_1+\ell t \le n-1$.
By adding them, we have $i_1+j_1+ks+\ell t \le 2(n-1)$.
From this and the equality $i_1+j_1 = ks+ \ell t$, we obtain the inequality
\begin{equation}\label{eq:ks+lt}
i_1+j_1=ks+\ell t \le n-1,	
\end{equation}
 i.e. $1+ks \le n- \ell t$.
Therefore $\{1+ks, 2+ks, \ldots, n-\ell t\} \neq \emptyset$ and so $\{1+ks - j_1, 2+ks-j_1, \ldots, n-\ell t - j_1\} \neq \emptyset$.
Since $i_1+ks \le n-1$,
\[(n-(i_1+j_1))-(1+ks-j_1)=n-1-(i_1+ks) \ge 0\]
and so $1+ks-j_1 \le n-(i_1+j_1)$.
Since $n-\ell t -j_1 \ge 1$, there is an integer $p \in \{1,2,\ldots, n-(i_1+j_1)\} \cap \{1+ks - j_1, 2+ks-j_1, \ldots, n-\ell t - j_1\}$.
Then
\begin{equation}\label{eq:set1}
1 \le p \le n-(i_1+j_1) 	
 \end{equation}
and
\begin{equation}\label{eq:set2}
1+ks-j_1 \le p \le n-\ell t -j_1	
 \end{equation}

Let $W=\{p,p+t, \ldots, p+\ell t, p+\ell t+s, \ldots, p+\ell t + ks\}$.
Then, by \eqref{eq:ks+lt} and \eqref{eq:set1},
\[1 \le p <p+t < \cdots <p+\ell t+s < \cdots < p+\ell t+ks=p+(i_1+j_1) < n.\]
Therefore $W \subset [n]$.
Now, by \eqref{eq:set2},
\[(p+qt)+j_1 \le p+\ell t+j_1 \le n\]
for each $q \in [\ell]$.
Thus $p+(q-1)t$ and $p+qt$ are adjacent for $q \in [\ell]$ by Lemma~\ref{lem:smallend}(i).
On the other hand, by \eqref{eq:set2},
\[(p+\ell t + (q-1)s)-i_1 \ge (1+ks-j_1)+\ell t-i_1=1\]
for each $q \in [k]$.
Thus, by Lemma~\ref{lem:smallend}(ii), $p+\ell t+(q-1)s$ and $p+\ell t + qs$ are adjacent for each $q \in [k]$.
Finally, $(p+ks+\ell t) -p = i_1+j_1$, so $p$ and $p+ks+\ell t$ are adjacent in  ${\rm RG}(T)$ by the same lemma.
Hence the subgraph of ${\rm RG}(T)$ induced by $W$ is a cycle of length $k+\ell+1$.
\end{proof}

\begin{theorem} For every integers $m$ and $n$ such that $4 \le m \le n$ and  $(m,n)\notin \{(4,4),(4,5),(6,8)\}$, there exists a Toeplitz matrix of order $n$ in ${\mathcal T_{\le 2}}$ whose row graph contains a cycle component of length $m$.\label{thm:cyclemn}
\end{theorem}
\begin{proof}
For integers $m$ and $n$ with $4 \le m \le n$ and $(m,n)\notin \{(4,4),(4,5),(6,8)\}$, we let
\[
T =
\begin{cases}  T_n \langle n-m+1, n-m+2; m-2,n-1 \rangle & \mbox{if $n < 2m-4$;}\\
T_n \langle 3,5;2m-8, 2m-7 \rangle & \mbox{if $n=2m-4$ and $m \ge 7$;}\\
T_6 \langle 1,4;3,5 \rangle & \mbox{if $n=6$ and $m=5$;}\\
T_n \langle 2, 4; 2m-6, 2m-5 \rangle & \mbox{if $n=2m-3$ and $m \ge 6$;}\\
T_7 \langle 1,5;4,6 \rangle & \mbox{if $n=7$ and $m=5$;}\\
T_n \langle m-2,n-1;n-m+1,n-m+2\rangle & \mbox{if $n >2m-3$.}
\end{cases}
\]
We first show that $T$ is well-defined in each case.
From the hypothesis $4\le m\le n$, it immediately follows that each of $n-m+1$, $n-m+2$, $m-2$, and $n-1$ is a positive integer. In the case $n=2m-4$ and $n=2m-3$, $2m-8 > 0$ and $2m-6 > 0$ since $m \ge 7$ and $m \ge 6$, respectively.

In addition, $n-m+1 < n-m+2 < m-2 < n-1 $ if $n < 2m-4$, and $m-2 < n-m+1 < n-m+2 < n-1$ if $n > 2m-3$.
It is easy to check that the four parameters for the cases $n=2m-4$ and $n=2m-3$ given above are distinct and arranged to fulfil the conditions which we imposed on a Toeplitz matrix.
Furthermore, by using Theorem~\ref{thm:ind2}, we can easily check that $T$ in each case belongs to $\mathcal{T}_{\le 2}$.

Now, we let $k=m-2$ and $\ell =1$ if $n < 2m-3$;
$k=m-5$ and $\ell=5$ if $n=2m-4$ and $m \ge 7$;
$k=m-4$ and $\ell=4$ if $n=2m-3$ and $m \ge 6$;
$k=1$ and $\ell = m-2$ if $n > 2m-3$. Then the integers $k$ and $\ell$ satisfy the inequalities and the equality given in Lemma~\ref{lem:klcycle}.
Thus $T$ contains a cycle of length $m$ for $4 \le m \le n$ except $(m,n)\in \{(4,4),(4,5),(5,6),(5,7),(6,8)\}$.
We may check that $152631$ and $162731$ are $5$-cycles in ${\rm RG}(T)$ for $(m,n)=(5,6)$ and $(m,n)=(5,7)$, respectively.
Since $T \in \mathcal{T}_{\le 2}$, an $m$-cycle is a component of ${\rm RG}(T)$ by Corollary~\ref{cor:starcomp}.
\end{proof}

Among the matrices in $\mathcal T_{\le 2}$ is of the form $T_n \langle i_1,i_2 ; j_1 \rangle$ or $T_n \langle i_1; j_1,j_2 \rangle$,
we may identify the matrices whose row graphs contain a cycle component.
We need the following lemmas.

\begin{lemma}\label{lem:iso}
 Let $G$ be the graph of a symmetric Toeplitz matrix $ST_n \langle u_1,\ldots, u_k \rangle$ with $\gcd(u_1,\ldots,u_k)=d \ge 2$.
 In addition, let $G_i$ be the subgraph of $G$ induced by the vertex set $\{v \in V(G) \mid v \equiv i \pmod{d}\}$ for each integer $1 \le i \le d$ and $r$ be an integer such that $r \equiv n\pmod{d}$ and $1\le r \le d$.
Then, for each integer $1 \le i \le d$,  $G_i$ is isomorphic to the graph $H_i$ of a symmetric Toeplitz matrix $T_i$ where
\begin{eqnarray*}
T_i = \begin{cases} ST_{\lceil n/d \rceil} \langle u_1/d, \ldots, u_k/d \rangle & \mbox{ if } 1 \le i \le r;\\
ST_{\lceil n/d \rceil -1} \langle u_1/d, \ldots, u_k/d \rangle  & \mbox{ if }r+1 \le  i \le d.
\end{cases}
\end{eqnarray*}
In addition, there is no edge joining a vertex in $G_i$ and a vertex in $G_j$ if $i \neq j$.
\end{lemma}
\begin{proof}
	Take two adjacent vertices $x$ and $y$ in $G$.
	Then the $(x,y)$-entry of a symmetric Toeplitz matrix $ST_n \langle u_1,\ldots, u_k \rangle$ is $1$, so $|x-y| \in \{u_1, \ldots,u_k\}$, Therefore $x \equiv y \pmod{d}$ by the hypothesis.
	Thus there is no edge joining a vertex in $G_i$ and a vertex in $G_j$ if $i \neq j$ by the definition of $G_i$.
	In addition, each of $G_1, \ldots, G_r$ has $\lceil n/d \rceil$ vertices and each of $G_{r+1}, \ldots, G_d$ has $\lceil n /d \rceil-1$ vertices.
	Take $i \in [r]$ and define a map
	$\varphi_i: V(G_i) \to V(H_i)$ sending $i+d \ell$ to $i+\ell$ (each element in $V(G_i)$ is in the form of $i+d\ell$ for an integer $\ell$).
	Then it is obvious that $\varphi_i$ is a bijection.
Now,
\begin{tabbing}
\= $x$ and $y$ are adjacent in  $G_i$ \\
\> $\Leftrightarrow$  $|x-y| \in \{u_1, \ldots, u_k\}$\\
\> $\Leftrightarrow$ $\left|(x-y)/d\right| \in \left\{u_1/d,\ldots, u_k/d\right\}$\\
\> $\Leftrightarrow$  $|\varphi_i(x)-\varphi_(y)| \in \left\{u_1/d,\ldots, u_k/d\right\}$ \\
\> $\Leftrightarrow$ $\varphi_i(x)$ and $\varphi_i(y)$ are adjacent in $T_{\lceil n/d \rceil} \langle u_1/d, \ldots, u_k/d \rangle$.
\end{tabbing}
	Thus $G_i \cong H_i$.
By a similar argument, we may show that $G_i \cong H_i$ for $i \in \{r+1, \ldots, d\}$.
	\end{proof}

\begin{lemma}\label{lem:iscycle}
Let $G$ be the graph of a symmetric Toeplitz matrix $ST_n\langle u_1,u_2 \rangle$ and let $d =\gcd(u_1,u_2) \ge 2$.
In addition, let $G_i$ be the subgraph of $G$ induced by the vertex set $\{v \in V(G) \mid v \equiv i \pmod{d}\}$ for each $i=1,2,\ldots,d$ and $r$ be an integer such that $r \equiv n\pmod{d}$ and  $1\le r \le d$.
Suppose that $n \le u_1+u_2 \le d\lceil n/d \rceil$.
Then $G_i$ is a cycle for each $i=1,2,\ldots,r$.
\end{lemma}
\begin{proof}
Fix $i \in [r]$.
By Lemma~\ref{lem:iso}, $G_i$ is the graph of $ST_{\lceil n/d \rceil} \langle u_1/d, u_2/d \rangle$.
By the hypothesis, \[ \frac{n}{d} \le \frac{u_1+u_2}{d} \le \left\lceil \frac{n}{d} \right\rceil,\] so $(u_1+u_2)/d = \lceil n/d \rceil$.
Therefore, by Lemma~\ref{lem:k2cycle}, there is a cycle $C_i$ of length $\lceil n/d \rceil$ in $G_i$.
Thus $C_i$ is a Hamilton cycle of $G_i$.
Take any vertex $v$ in $G_i$.
Then each neighbor of $v$ belongs to $\{v+u_2, v+u_1, v-u_2, v-u_1\}$.
Since $u_2+u_1 \ge n$ and $(v+u_2)-(v-u_1) = u_1+u_2 \ge n$, at most one of $v+u_2$ and $v-u_1$ can be a vertex of $G$.
For the same reason, at most one of $v+u_1$ and $v-u_2$ can be a vertex of $G$.
Thus the degree of $v$ is at most two.
Since $v$ was arbitrarily chosen, each vertex in $G_i$ has degree at most $2$.
Therefore $C_i$ itself is $G_i$.
\end{proof}

\begin{lemma}\label{lem:acyclic}
Let $G$ be the graph of a symmetric Toeplitz matrix $ST_n\langle u_1,u_2 \rangle$ where $\gcd(u_1,u_2)=1$. If $u_1+u_2 >n$, then $G$ is acyclic.
\end{lemma}
\begin{proof}
Suppose, to the contrary, that $u_1+u_2 >n$ and $G$ contains a cycle $v_1v_2\cdots v_\ell v_1$.
Fix $i \in [\ell]$.
Then $|v_{i+1}-v_i|=u_1$ or $u_2$  (we identify $v_{\ell+1}$ and $v_{\ell+2}$  with $v_1$ and $v_2$, respectively).
Suppose that $v_{i+1}-v_i =\pm u_1$.
If $v_{i+2}-v_{i+1}=\mp u_1$, then  $v_{i+2}-v_i=0$, a contradiction.
If $v_{i+2}-v_{i+1}=\pm u_2$, then
$|v_{i+2}-v_i|=u_1+u_2 >n$, a contradiction again.
Thus, if $v_{i+1}-v_{i}=\pm u_1$, then $v_{i+2}-v_{i+1}=\pm u_1$ or $\mp u_2$.
By a similar argument, we may show that if $v_{i+1}-v_{i}=\pm u_2$, then $v_{i+2}-v_{i+1}=\mp u_1$ or $\pm u_2$.	
Therefore $v_{i+1}-v_i \in \{u_1,-u_2\}$ for each $i \in [\ell]$ or $v_{i+1}-v_i \in \{-u_1,u_2\}$ for each $i \in [\ell]$.
We assume the former.
Let $n_1$ and $n_2$ be the numbers of $i$ with $v_{i+1}-v_i=u_1$ and $v_{i+1}-v_i=-u_2$, respectively.
Then $n_1u_1+n_2(-u_2)=0$, so $n_1u_1 = n_2u_2$.
Since $\gcd(u_1,u_2)=1$, $n_1 \ge u_2$ and $n_2 \ge u_1$ and so $n_1+n_2 \ge u_1+u_2 >n$ and we reach a contradiction.
By a similar argument, we should reach a contradiction even in the latter case.
\end{proof}

Now we are ready to characterize the matrices of the form $T_n \langle i_1,i_2 ; j_1 \rangle$ whose row graph contains a cycle.

\begin{theorem}\label{thm:s1t2cycle}
Let $T=T_n\langle i_1,i_2;j_1 \rangle$ be a Toeplitz matrix in ${\mathcal T}_{\le 2}$, $d = \gcd(i_1+j_1,i_2-i_1)$, and
$r$ be an integer such that $1 \le r \le d$ and $n \equiv r \pmod{d}$.
Then the row graph of $T$ contains a cycle if and only if
$i_2 \ne 2i_1+j_1$, $i_2+j_1 \le d \lceil n/d \rceil$ and $1+i_1 \le r$.
Especially, if the row graph contains a cycle, then its length is $\lceil n/d \rceil$.
\end{theorem}
\begin{proof}
Let $G={\rm RG}(T)$.
Since $T$ belongs to ${\mathcal T}_{\le 2}$ by the hypothesis, $i_2+j_1 \ge n$ by Theorem~\ref{thm:ind2}.
Therefore $G$ is a spanning subgraph of the graph $L$ of $ST_n \langle i_1+j_1, i_2-i_1 \rangle$ by Proposition~\ref{prop:inequality}.
Let $G_i$ denote the subgraph of $L$ induced by the vertex set $\{v \in V(G) \mid v \equiv i \pmod{d}\}$ for each $i=1,\ldots,d$.

To show the `only if' part, suppose that $G$ contains a cycle $C$.
Then $L$ contains $C$.
By Lemma~\ref{lem:iso},  $C$ is contained in $G_t$ for some $t \in [d]$.
By the same lemma, $G_t$ is isomorphic to the graph of  $ST_{\lceil n/d \rceil} \langle (i_1+j_1)/d,(i_2-i_1)/d \rangle$ or $ST_{\lfloor n/d \rfloor} \langle (i_1+j_1)/d,(i_2-i_1)/d \rangle$.
Note that \[\left\lceil \frac{n}{d}\right\rceil -1 < \frac{n}{d} \le \frac{i_2+j_1}{d} = \frac{i_1+j_1}{d}+\frac{i_2-i_1}{d}.\]
Therefore, if $G_t$ is isomorphic to the graph of $ST_{\left\lceil n/d\right\rceil -1} \langle(i_1+j_1)/d,(i_2-i_1)/d\rangle$, then $G_t$ is acyclic by Lemma~\ref{lem:acyclic} and we reach a contradiction.
Thus $G_t$ is isomorphic to the graph of $ST_{\left\lceil n/d \right \rceil} \langle (i_1+j_1)/d,(i_2-i_1)/d\rangle$.
Hence $t \in [r]$  by Lemma~\ref{lem:iso}.

Suppose, to the contrary, that $i_2+j_1 > d \lceil n/d \rceil$.
Then
\[ \left\lceil \frac{n}{d} \right\rceil < \frac{i_2+j_1}{d}=\frac{i_1+j_1}{d}+\frac{i_2-i_1}{d}.\]
Therefore, by Lemma~\ref{lem:acyclic},
the graph of $ST_{\lceil n/d \rceil}\langle (i_1+j_1)/d,(i_2-i_1)/d \rangle$ is acyclic and we reach a contradiction.
Thus $i_2 + j_1 \le d\lceil n/d \rceil$.
Then \[ \frac{n}{d} \le \frac{i_2+j_1}{d} \le \left\lceil \frac{n}{d} \right\rceil,\] so $(i_2+j_1)/d = \lceil n/d \rceil$.
Thus, by Lemma~\ref{lem:iscycle}, $G_t$ is a cycle.
Hence, $G_t = C$ and the length of $C$ is $(i_2+j_1)/d$.

Since $G_t$ is a cycle, $t$ is adjacent to two vertices in $G_t$.
Since $G$ is the row graph of $T_n\langle i_1,i_2;j_1\rangle$, $t+i_2-i_1$, $t+i_1+j_1$ and $t+i_2+j_1$ are only possible neighbors of $t$.
Since $i_2+j_1 \ge n$, $t+i_2-i_1$ and $t+i_1+j_1$ are the two neighbors of $t$.
Now, if $1+i_1 >r$, then $1 \le i \le r<1+i_1$ and so, by Lemma~\ref{lem:smallend}, $i$ and $i+i_2-i_1$ cannot be adjacent in $G$.
Thus $1+i_1 \le r$ and we have shown the `only if' part.

To show the `if' part, suppose $i_2+j_1 \le d\lceil n/d \rceil$ and $1+i_1 \le r$.
Let $H$ be the subgraph of $G$ induced by the vertex set $\{v \in V(G) \mid v \equiv i_1+1 \pmod d\}$.
Since $1+i_1 \le r$, $H$ has $\lceil n/d \rceil$ vertices and $1+i_1$ is the vertex with smallest labeling.
We show that $H = G_{i_1+1}$.
By the definition, $V(H) = V(G_{i_1+1})$.
Since $G$ is a subgraph of $L$, $H$ is a subgraph of $G_{i_1+1}$.
Take two adjacent vertices $x$ and $y$ in $G_{i_1+1}$.
Then $|x-y| = i_2-i_1$ or $i_1+j_1$.
If $|x-y| = i_1+j_1$, $x$ adn $y$ are adjacent in $G$ and so in $H$.
Suppose $|x-y|=i_2-i_1$.
Without loss of generality, we may assume $x-y = i_2-i_1$.
Since $y \in V(H)$, $y \ge 1+i_1$.
Thus $x$ and $y$ are adjacent in $G$ by Lemma~\ref{lem:smallend} and so in $H$.
Therefore we have shown $H = G_{i_1+1}$.
By Lemma~\ref{lem:iscycle}, $G_{i_1+1}$ is a cycle.
Thus $H$ is also a cycle, which implies that $G$ has a cycle.
\end{proof}

Since the row graphs of $T_n \langle \alpha; \beta \rangle$ and $T_n \langle \beta; \alpha \rangle$ are isomorphic, we may obtain the following result as well.

\begin{corollary}
Let $T=T_n\langle i_1;j_1,j_2 \rangle$ be a Toeplitz matrix in ${\mathcal T}_{\le 2}$, $d = \gcd(i_1+j_1,j_2-j_1)$, and
$r$ be an integer such that $1 \le r \le d$ and $n \equiv r \pmod{d}$.
Then the row graph of $T$ contains a cycle if and only if
$j_2 \ne i_1+2j_1$, $i_1+j_2 \le d \lceil n/d \rceil$ and $1+j_1 \le r$.
Especially, if the row graph contains a cycle, then its length is $\lceil n/d \rceil$.
\end{corollary}

In the next subsection, we focus on the Toeplitz matrices in ${\mathcal T}_{\le 2}^{\star}$ whose row graphs consist of only cycle components.
We denote the set of such matrices by ${\mathcal T}_{\le 2}^{\circ}$.

\subsection{Toeplitz matrices in  ${\mathcal T}_{\le 2}^{\circ}$}

We first give a sufficient condition for a Toeplitz matrix belonging to ${\mathcal T}_{\le 2}^{\circ}$.
Before doing so, we need to deduce several results.

Recall that ${\mathcal C}(T)$ denotes the adjacency matrix of ${\rm RG}(T)$ for a Toeplitz matrix $T$ and $ST_n \langle \gamma \rangle$ means a symmetric Toeplitz matrix $T_n \langle \gamma, \gamma \rangle$.
If ${\mathcal C}(T)$ is a Toeplitz matrix for a Toeplitz matrix $T$ of order $n$, then ${\mathcal C}(T)=ST_n \langle \gamma \rangle$ for some $\gamma\subset [n-1]$ since ${\mathcal C}(T)$ is symmetric.

For $\alpha=\{i_1,\ldots,i_{k_1}\}$ and $\beta=\{j_1,\ldots,j_{k_2}\}$ where $k_1,k_2\ge2$, define the following sets based upon Lemma~\ref{lem:smallend}:
\begin{align}\label{cor:cdsubg}
\gamma_1&= \{i_t-i_s\in[n-1] \mid 1 \le s < t \le k_1 \}; \notag\\
\gamma_2 &= \{j_t-j_s\in[n-1] \mid 1 \le s < t \le k_2\};  \\
\gamma_3 &= \{i_s+j_t\in[n-1] \mid 1 \le s \le k_1, 1 \le t \le k_2\}.\notag \end{align}
where $\gamma_1=\emptyset$ (resp.\ $\gamma_2= \emptyset$) if $k_1\le1$ (resp. $k_2\le1$) and $\gamma_3 = \emptyset$ if $\alpha=\emptyset$ or $\beta=\emptyset$.

For convenience, we let
\[\gamma(\alpha,\beta)=\gamma_1 \cup \gamma_2 \cup \gamma_3.\]
We denote by $A\le B$ for matrices $A=[a_{ij}]$ and $B=[b_{ij}]$ of order $n$ such that $a_{ij}\le b_{ij}$ for all $i,j\in[n]$.
The following result immediately follows from Lemma~\ref{lem:smallend}.
\begin{proposition}\label{prop:inequality}
For $T=T_n \langle \alpha;\beta\rangle$, the row graph of $T$ is a spanning subgraph of the graph of $ST_n \langle \gamma(\alpha,\beta)\rangle$, i.e. ${\mathcal C}(T) \le ST_n \langle \gamma(\alpha,\beta)\rangle$.
\end{proposition}

\begin{proposition}\label{prop:cdistoep}
Let $T=T_n \langle i_1, i_2 ; j_1, j_2 \rangle$ be a Toeplitz matrix with $i_1+j_2=i_2+j_1=n$ and $2(i_1+j_1) \neq n$.
Then the row graph of $T$ is the graph of $ST_n \langle i_1+j_1,n-(i_1+j_1)\rangle$ and consists of only cycle components.
\end{proposition}

\begin{proof}
Since $i_1+j_2=i_2+j_1=n$, ${\mathcal C}(T) \le ST_n \langle i_2-i_1, j_2-j_1, i_1+j_1  \rangle$ by \eqref{cor:cdsubg} and Proposition~\ref{prop:inequality}.
By the hypothesis, $j_2-j_1 = i_2-i_1 = n-(i_1+j_1)$, so $ST_n \langle i_2-i_1, j_2-j_1, i_1+j_1 \rangle = ST_n \langle i_1+j_1, n-(i_1+j_1) \rangle$.
To show ${\mathcal C}(T)=ST_n \langle i_1+j_1, n-(i_1+j_1)\rangle$, take an entry $1$ of $ST_n \langle i_1+j_1, n-(i_1+j_1) \rangle$.
Let $(v,w)$ be the address of the entry.
By symmetry, we may assume $v<w$.
Then $w-v = i_1+j_1$ or $w-v=n-(i_1+j_1)$ by the definition of Toeplitz matrix.

If $w-v = i_1+j_1$, then the $(v,w)$-entry of ${\mathcal C}(T)$ is $1$ by Lemma~\ref{lem:smallend}(iii).
Now suppose $w-v=n-(i_1+j_1)$.

{\it Case 1.} $w \le n-j_1$.
Then $j_1 \le n-w$.
In addition, $v=w-(n-j_1)+i_1 \le i_1=n-j_2$, so $j_2 \le n-v$.
Since $w-v=n-(i_1+j_1)=j_2-j_1$, the $(v,w)$-entry of ${\mathcal C}(T)$ is $1$ by Lemma~\ref{lem:smallend}(i).

{\it Case 2.} $w > n-j_1$.
Then $w >i_2$ and $v = w-(n-j_1)+i_1 > i_1$.
Since $w-v=n-(i_1+j_1)=i_2-i_1$,
the $(v,w)$-entry is $1$ by Lemma~\ref{lem:smallend}(ii).

In both cases, we have shown that ${\mathcal C}(T) = ST_n \langle i_1+j_1, n-(i_1+j_1) \rangle$.
It is known that the number of $1$'s in a symmetric Toeplitz matrix $ST_n \langle u_1, \ldots, u_k \rangle$ is
$2(kn-\sum_{i=1}^ku_i)$.
Thus $ST_n \langle i_1+j_1, n-(i_1+j_1) \rangle$ has exactly $2n$ entries of $1$ and so ${\mathcal C}(T)$ has $2n$ entries of $1$, which implies that each row sum is exactly $2$.
Hence ${\mathcal C}(T)$ consists of only cycle components.
\end{proof}

\begin{lemma}[\cite{CJKKM}]\label{lem:k2cycle} Let $G$ be the graph of a symmetric Toeplitz matrix $ST_n\langle u_1, u_2\rangle$ with $n \ge u_1+u_2$.
Then $G$ contains a cycle of length $(u_1+u_2)/ \gcd(u_1,u_2)$.\end{lemma}

In the previous section, we showed that a cycle of length at least $5$ is a Toeplitz-row graph.
Now we identify the Toeplitz matrices each of whose row graph is a cycle.

\begin{theorem}\label{thm:cycle}
The row graph of a Toeplitz matrix $T$ is a cycle if and only if $T=T_n \langle i_1, i_2 ; j_1, j_2 \rangle$ with $i_1+j_2=i_2+j_1 = n$  and $\gcd (n, i_1+j_1)  = 1$.
\end{theorem}
\begin{proof}
To show the `if' part, let $T=T_n \langle i_1, i_2 ; j_1, j_2 \rangle$ with $i_1+j_2=i_2+j_1 = n$  and $\gcd (n, i_1+j_1)  = 1$.
We denote the row graph of $T$ by $G$.
By Proposition~\ref{prop:cdistoep}, $G$ consists of cycle components.
Since $\gcd(n, i_1+j_1) = 1$, $n-(i_1+j_1) \neq i_1+j_1$.
Therefore, by Lemma~\ref{lem:k2cycle}, $G$ contains a cycle of length $n$.
Thus $G$ itself is a cycle of length $n$.

To show the `only if' part, suppose that the row graph $G$ of a Toeplitz matrix $T\langle \alpha, \beta \rangle$ is a cycle of length $n$.
By Proposition~\ref{prop:cyclepath}, $n \ge 5$.
By Corollary~\ref{cor:trinaglefree}, the maximum row sum of $T$ is at most $2$.
Suppose that $|\alpha|=1$.
Then the vertex $n$ is adjacent to only $n-i_1-j_1$ by the hypothesis that $i_1+j_2=n$ and Lemma~\ref{lem:smallend}(i) and (iii).
Thus $n$ has degree on in $G$.
By a similar argument, one can show that $1$ is adjacent to only $1+i_1+j_1$ if $|\beta|=1$.
In both cases, we have reached a contradiction.
Thus $|\alpha|=|\beta|=2$.
By  Lemma~\ref{lem:smallend},
$G$ has at most \[(n-j_2)+(n-i_2)+(n-j_1-i_1)=3n-(i_1+j_2)-(i_2+j_1)\]
edges.
Since $G$ is a cycle of length $n$,
\[3n-(i_1+j_2)-(i_2+j_1) \ge n.\]
On the other hand, since the row sum of $T$ is at most $2$, $i_1+j_2 \ge n$ and $i_2+j_1 \ge n$ by Theorem~\ref{thm:ind2}.
Therefore $i_1+j_2 = i_2+j_1 = n$.
Thus, if $w$ and $x$ are adjacent in $G$, then $|w-x| \in \{i_2-i_1,j_2-j_1,i_1+j_1\}$ by Lemma~\ref{lem:smallend}.
Now let $d = \gcd (n, i_1+j_1)$.
Since $i_2-i_1 =j_2-j_1 = n-(i_1+j_1)$, $d$ is a common divisor of $i_2-i_1$, $j_2-j_1$, and $i_1+j_1$.
Therefore, if $w$ and $x$ are adjacent in $G$, then $w \equiv x \pmod d$.
If $d >1$, then $1 \not \equiv 2 \pmod d$ and $G$ has at least two components, which is impossible.
Thus $d = \gcd (n, i_1+j_1) = 1$.
\end{proof}

By Theorem~\ref{thm:forbidden}, any triangle-free Toeplitz-row graph cannot have cycles of three different lengths.
Yet, we have not found a triangle-free Toeplitz-row graph with cycles of two distinct lengths.
The following theorem finds ones whose row graphs consist of cycles of the same length among such Toeplitz matrices, which extends Theorem~\ref{thm:cycle}.
\begin{theorem}
Let $T= T_n \langle i_1, i_2 ; j_1, j_2 \rangle$ be a Toeplitz matrix with $i_1+j_2=i_2+j_1=n$ and $2(i_1+j_1) \neq n$. Then the row graph of $T$ is a disjoint union of exactly $d$ cycles of length $n/d$ where $d = \gcd(n, i_1+j_1)$.\label{thm:dcycle}
\end{theorem}
\begin{proof}
Let $G$ be the row graph of $T$.
 By Proposition~\ref{prop:cdistoep},  $G$ is the graph of $ST_n \langle i_1+j_1,n-(i_1+j_1)\rangle$ and consists of only cycle components.
Let $V_i = \{i,i+d,\ldots,  i+(n/d -1) d\}$ for each $i \in [d]$.
Then, by Lemma~\ref{lem:iso}, the subgraph $G_i$ of $G$ induced by $V_i$ is isomorphic to the graph of $ST_{n/d} \langle (i_1+j_1) /d , (n-i_1-j_1)/d \rangle$ and there is no edge joining a vertex in $G_i$ and a vertex in $G_j$ if $i \neq j$.
By Theorem~\ref{thm:cycle}, $ST_{n/d} \langle (i_1+j_1) /d , (n-i_1-j_1)/d \rangle$ is a cycle of length $n/d$ and this completes the proof.
\end{proof}

\end{document}